\documentclass[reqno,10pt,centertags]{amsart}
\usepackage{amsmath,amsthm,amscd,amssymb,latexsym,upref,stmaryrd}
\usepackage{amsfonts,mathrsfs}
\usepackage{hyperref} 

\newcommand*{\mailto}[1]{\href{mailto:#1}{\nolinkurl{#1}}}

\newcommand{\R}{{\bbR}}

\newcommand{\bbC}{{\mathbb{C}}}

\newcommand{\bbN}{{\mathbb{N}}}

\newcommand{\bbR}{{\mathbb{R}}}
\newcommand{\bbS}{{\mathbb{S}}}

\newcommand{\cB}{{\mathcal B}}
\newcommand{\cC}{{\mathcal C}}

\newcommand{\cH}{{\mathcal H}}

\newcommand{\cM}{{\mathcal M}}

\newcommand{\cX}{{\mathcal X}}

\newcommand{\gQ}{\mathfrak Q}

\newcommand{\gq}{\mathfrak q}

\DeclareMathOperator{\dom}{dom}

\renewcommand{\Re}{\text{\rm Re}}

\renewcommand{\ln}{\text{\rm ln}}

\newcommand{\dott}{\,\cdot\,}

\newcommand{\loc}{\text{\rm{loc}}}

\newcommand{\beq}{\begin{equation}}
\newcommand{\enq}{\end{equation}}

\newcommand{\Om}{\Omega}
\newcommand{\dOm}{{\partial\Omega}}

\newcommand{\LOm}{L^2(\Om;d^nx)}
\newcommand{\LdOm}{L^2(\dOm;d^{n-1} \omega)}

\newcommand{\ga}{\gamma}

\newcommand{\no}{\notag}
\newcommand{\lb}{\label}
\newcommand{\f}{\frac}

\newcommand{\ol}{\overline}

\newcommand{\wti}{\widetilde}

\newcommand{\bi}{\bibitem}
\newcommand{\prodm}{\mu \otimes \mu}

\renewcommand{\ge}{\geqslant}
\renewcommand{\le}{\leqslant}

\let\geq\geqslant
\let\leq\leqslant

\makeatletter
\def\theequation{\@arabic\c@equation}

\allowdisplaybreaks 
\numberwithin{equation}{section}

\newtheorem{theorem}{Theorem}[section]
\newtheorem{proposition}[theorem]{Proposition}
\newtheorem{lemma}[theorem]{Lemma}

\newtheorem{definition}[theorem]{Definition}
\newtheorem{hypothesis}[theorem]{Hypothesis}

\theoremstyle{remark}
\newtheorem{remark}[theorem]{Remark}

\begin{document}

\title[Heat Kernel Bounds for Elliptic PDEs in Divergence Form]{Heat Kernel 
Bounds for Elliptic Partial Differential Operators in Divergence Form with Robin-Type Boundary Conditions}

\author[F.\ Gesztesy]{Fritz Gesztesy}
\address{Department of Mathematics,
University of Missouri, Columbia, MO 65211, USA}
\email{\mailto{gesztesyf@missouri.edu}}
%\email{gesztesyf@missouri.edu}
\urladdr{\url{http://www.math.missouri.edu/personnel/faculty/gesztesyf.html}}
%\urladdr{http://www.math.missouri.edu/personnel/faculty/gesztesyf.html}

\author[M.\ Mitrea]{Marius Mitrea}
\address{Department of Mathematics,
University of Missouri, Columbia, MO 65211, USA}
\email{\mailto{mitream@missouri.edu}}
%\email{mitream@missouri.edu}
\urladdr{\url{http://www.math.missouri.edu/personnel/faculty/mitream.html}}
%\urladdr{\href{http://www.math.missouri.edu/personnel/faculty/mitream.html}

\author[R.\ Nichols]{Roger Nichols}
\address{Mathematics Department, The University of Tennessee at Chattanooga, 
415 EMCS Building, Dept. 6956, 615 McCallie Ave, Chattanooga, TN 37403, USA}
\email{\mailto{Roger-Nichols@utc.edu}}
%\email{Roger-Nichols@utc.edu}
\urladdr{\url{http://rogeranichols.jimdo.com/}} 
%\urladdr{http://rogeranichols.jimdo.com/}

\author[E.\ M.\ Ouhabaz]{El Maati Ouhabaz}
\address{University of Bordeaux, Institut de Math\'ematiques (IMB), Equipe d'Analyse, 351, 
Cours de la Lib\'eration, 33405 Talence, France}
\email{\mailto{Elmaati.Ouhabaz@math.u-bordeaux1.fr}}
%\email{Elmaati.Ouhabaz@math.u-bordeaux1.fr}
\urladdr{\url{http://www.math.u-bordeaux1.fr/~eouhabaz/}}
%\urladdr{http://www.math.u-bordeaux1.fr/~eouhabaz/}

\dedicatory{Dedicated with great pleasure to E.\ Brian Davies on the occasion of his 70th birthday.}
\thanks{Work of M.\,M.\ was partially supported by the Simons Foundation Grant $\#$\,281566.} 
\thanks{Work of  E.\,M.\,O.\ was partly supported by the ANR project  "Harmonic Analysis at its Boundaries",  ANR-12-BS01-0013-02.}
%\thanks{To appear in {\it .}}
\date{\today}
\subjclass[2010]{Primary 35J15, 35J25, 47D06; Secondary 46E35, 47A10, 47D07.}
\keywords{Positivity preserving semigroups, elliptic partial differential operators, 
Robin boundary conditions, heat kernel bounds, Green's function bounds.}

%%%%%%%%%%%%%%%%%%%%%%%%%%%%%%%%%%%%%%
%%%%%%%%%%%%%%%%%%%%%%%%%%%%%%%%%%%%%%
\begin{abstract} 
The principal aim of this short note is to extend a recent result on Gaussian heat kernel 
bounds for self-adjoint $L^2(\Om; d^n x)$-realizations, $n\in\bbN$, $n\geq 2$, 
of divergence form elliptic partial differential expressions $L$ with (nonlocal) Robin-type boundary conditions in bounded Lipschitz domains $\Om \subset \bbR^n$, where
$$
Lu = - \sum_{j,k=1}^n\partial_j a_{j,k}\partial_k u. 
$$
The (nonlocal) Robin-type boundary conditions are then of the form 
$$
\nu\cdot A\nabla u + \Theta\big[u\big|_{\partial\Om}\big]=0 \, \text{ on } \, \partial\Omega,
$$ 
where $\Theta$ represents an appropriate operator acting on Sobolev spaces associated with the 
boundary $\partial \Om$ of $\Om$, and $\nu$ denotes the outward pointing normal unit vector on 
$\partial\Om$. 
\end{abstract}
%%%%%%%%%%%%%%%%%%%%%%%%%%%%%%%%%%%%%%
%%%%%%%%%%%%%%%%%%%%%%%%%%%%%%%%%%%%%%

\maketitle

%{\scriptsize{\tableofcontents}}

%%%%%%%%%%%%%%%%%%%%%%%%%%%%%%%%%%%%%%
%%%%%%%%%%%%%%%%%%%%%%%%%%%%%%%%%%%%%%
\section{Introduction}  \lb{s1}
%%%%%%%%%%%%%%%%%%%%%%%%%%%%%%%%%%%%%%
%%%%%%%%%%%%%%%%%%%%%%%%%%%%%%%%%%%%%%

This  note represents an addendum to the recent paper \cite{GMN13} devoted to a 
new class of self-adjoint realizations $L_{\theta, \Om}$ in 
$L^2(\Om; d^n x)$ of elliptic partial differential expressions in divergence form, 
\begin{equation} 
L = - \sum_{j,k=1}^n\partial_ja_{j,k}\partial_k,   \lb{1.0} 
\end{equation} 
on bounded Lipschitz domains 
$\Om \subset \bbR^n$, $n\geq 2$, with Robin boundary conditions of the form 
$\nu \cdot A\nabla u + \theta\big(u\big|_{\partial\Om}\big)= 0$. 
(Here $\nu$ denotes the outward pointing 
normal unit vector and $\theta$ is a suitable function on the boundary $\partial \Om$ of $\Om$.) Following \cite{GM09}, we put particular emphasis in \cite{GMN13} on developing a theory of nonlocal Robin boundary conditions where the function $\theta$ on $\partial \Om$ is replaced by a suitable operator $\Theta$ acting in $L^2(\partial \Om; d^{n-1} \omega)$, with $d^{n-1} \omega$ representing the surface measure on $\partial \Om$. 
(More precisely, $\Theta$ acts in appropriate Sobolev spaces 
on the boundary of $\Om$). The resulting self-adjoint operator in $L^2(\Om; d^n x)$ 
is then denoted by $L_{\Theta, \Om}$ and we study its resolvent and semigroup, proving a Gaussian heat kernel bound and a bound for the Green's function of $L_{\Theta, \Om}$. 

To keep this note short, we will refer the reader to the detailed paper \cite{GMN13}, especially, 
we refer to the extensive introduction and long list of references contained therein. In particular, we will only reproduce that material from \cite{GMN13} that is absolutely necessary to read this note. 

In Section \ref{s2} we provide a bit of background and restate the principal result of \cite{GMN13} 
and then our current improvement based on a natural, additional condition. Section \ref{s3} then provides some concrete illustrations. 

Finally, we briefly summarize some of the notation used in this paper: 
Let $\cH$ be a separable complex Hilbert space, $(\cdot,\cdot)_{\cH}$ the scalar product in $\cH$
(linear in the second argument), and $I_{\cH}$ the identity operator in $\cH$.

Next, if $T$ is a linear operator mapping (a subspace of) a Hilbert space into another, then 
$\dom(T)$ and $\ker(T)$ denote the domain and kernel (i.e., null space) of $T$. 
The closure of a closable operator $S$ is denoted by $\ol S$. 
The spectrum, essential spectrum, discrete spectrum, and resolvent set 
of a closed linear operator in a Hilbert space will be denoted by $\sigma(\cdot)$, 
$\sigma_{\rm ess}(\cdot)$, $\sigma_{\rm d}(\cdot)$, and $\rho(\cdot)$, respectively. 

The Banach spaces of bounded and compact linear operators on a separable complex Hilbert space 
$\cH$ are denoted by $\cB(\cH)$ and $\cB_\infty(\cH)$, respectively. 
The analogous notation $\cB(\cX_1,\cX_2)$, $\cB_\infty (\cX_1,\cX_2)$ will be used for 
bounded and compact operators between two Banach spaces $\cX_1$ and $\cX_2$. 

Given a $\sigma$-finite measure space, $(M,\cM,\mu)$, the product measure on 
$M \times M$ will be denoted by $\prodm$. Without loss of generality, we also denote the completion 
of the product measure space $(M\times M, \cM \otimes \cM, \prodm)$ by the same symbol and always 
work with this completion in the following. 

For $a, b \in \bbC^n$, we use the Euclidean pairing $\langle a, b \rangle_{\bbC^n} = 
\sum_j a_j b_j = (\ol a, b)_{\bbC^n}$.

%%%%%%%%%%%%%%%%%%%%%%%%%%%%%%%%%%%%%%
%%%%%%%%%%%%%%%%%%%%%%%%%%%%%%%%%%%%%%
\section{Nonlocal Robin boundary conditions}\lb{s11}
%\section{Gaussian Heat Kernel Bounds}  \lb{s2}
%%%%%%%%%%%%%%%%%%%%%%%%%%%%%%%%%%%%%%
%%%%%%%%%%%%%%%%%%%%%%%%%%%%%%%%%%%%%%

We start by recalling our basic notation on positivity preserving/improving operators. 

%%%%%%%%%%
\begin{hypothesis} \lb{h2.1}
Let $(M,\cM,\mu)$ denote a $\sigma$-finite, separable measure space associated with a nontrivial positive
measure $($i.e., $0 < \mu(M) \leq \infty$$)$.  
\end{hypothesis}
%%%%%%%%%%

The set of nonnegative elements $0 \leq f \in L^2(M; d\mu)$ (i.e., $f(x) \geq 0$ $\mu$-a.e.) is a cone 
in $L^2(M; d\mu)$, closed in the norm and weak topologies. 

%%%%%%%%%%
\begin{definition} \lb{d2.2} 
Let $A$ be a bounded linear operator in $L^2(M; d\mu)$. Then $A$ is called {\it positivity preserving} 
$($resp., {\it positivity improving}$)$ if
\begin{equation}
0 \neq f \in L^2(M;d\mu), \, f \geq 0 \text{ implies } \, A f \geq 0 \, 
\text{ $($resp., $Af > 0$$)$}.
\end{equation}
Given two bounded operators $A$ and $B$ on $L^2(M; d\mu)$ such that $A$ is positivity preserving, we say that $B$ is dominated by $A$ if 
\begin{equation}
  | Bf | \le A |f |,\quad  f \in L^2(M;d\mu). 
 \end{equation}
\end{definition}
%%%%%%%%%%
Here and in the rest of this paper, all the inequalities (and equalities) are understood in the $\mu$-a.e. sense. 

Turning our attention to integral operators in $L^2(M; d\mu)$ with associated integral kernels 
$A(\cdot,\cdot)$ on $M \times M$, we assume that
\begin{equation}
\text{$A(\cdot,\cdot): M \times M \to \bbC$ is $\prodm$-measurable,} 
\end{equation}  
and introduce the integral operator $A$ associated with the 
integral kernel $A(\cdot,\cdot)$ as follows: 
\begin{equation} 
(A f)(x):= \int_M A(x,y) f(y) \, d\mu(y) \text{ for $\mu$-a.e.\ $x\in M$, } \, f\in L^2(M; d\mu). 
\end{equation} 
This means that $ A(x,\cdot) f(\cdot)$ is absolutely integrable over $M$ for $\mu$-a.e.\ $x \in M$ and 
$\int_M A(\cdot,y) f(y) \, d\mu(y) \in L^2(M; d\mu)$. 

Suppose that $A$ is bounded on $L^2(M;d\mu)$. Then it is a classical fact that $A$ is positivity preserving 
if and only if 
\begin{equation}
A(\cdot,\cdot) \geq 0 \;\, \prodm \text{-a.e.\ on } M\times M.    \lb{2.13} 
\end{equation} 
Similarly, if $B(x,y)$ denotes the integral kernel of an integral operator $B$
that is bounded on $L^2(M;d\mu)$, then $B$ is dominated by $A$ if and only if 
\begin{equation}
| B(\cdot,\cdot) | \le A(\cdot,\cdot)  \;\, \prodm \text{-a.e.\ on } M\times M.    \lb{2.13d} 
\end{equation} 

Next we briefly turn to the basics for divergence form elliptic partial differential 
operators with (nonlocal) Robin-type boundary conditions in $n$-dimensional, 
bounded, Lipschitz domains, corresponding to differential expressions $L$ given by 
\begin{equation}\label{L-def1}
Lu:= - \sum_{j,k=1}^n\partial_j a_{j,k} \partial_k u. 
\end{equation}

For basic facts on Sobolev spaces on $\Om$ or  $\dOm$ and Dirichlet and Neumann trace operators, as well as the choice of notation used below, we refer to \cite[Appendix\ A]{GMN13}. For the basics on sesquilinear forms and operators associated with them we refer to \cite {Da89}, \cite {Ou05} and  \cite[Appendix\ B]{GMN13}. 

In the remainder of this section we make the following assumption:

%%%%%%%%%%%
\begin{hypothesis} \lb{h4.1} Let $n\in\bbN$, $n\geq 2$. \\
$(i)$ Assume that $\Om\subset{\bbR}^n$ is a bounded Lipschitz domain. \\
$(ii)$ Suppose that the matrix 
\begin{equation}\label{igc-Ts.55}
A(\cdot) = (a_{j,k}(\cdot))_{1\leq j,k\leq n}
\end{equation} 
satisfies $A\in L^{\infty}(\Om; d^nx)^{n\times n}$ and is real symmetric 
a.e.\ on $\Om$. In addition, given $0 < a_0 < a_1 < \infty$, assume that $A$ satisfies the uniform ellipticity conditions
\begin{equation}
a_0 I_n \leq A(x) \leq a_1 I_n \, \text{ for a.e.\ $x \in \Om$.}
\end{equation} 
\end{hypothesis}
%%%%%%%%%%%

Above $I_n$ represents the identity matrix in $\bbC^n$ and we will denote the identity operators 
in $\LOm$ and $\LdOm$ by $I_{\Om}$ and $I_{\dOm}$, respectively. Also, in the sequel, the
sesquilinear form 
\begin{equation}
\langle \dott, \dott \rangle_{s} ={}_{H^{s}(\dOm)}\langle\dott,\dott
\rangle_{H^{-s}(\dOm)}\colon H^{s}(\dOm) \times H^{-s}(\dOm)
\to \bbC, \quad s\in [0,1],
\end{equation} 
(antilinear in the first, linear in the second factor), will denote the duality
pairing between $H^s(\dOm)$ and 
\begin{equation}\lb{3.2}
H^{-s}(\dOm)=\big(H^s(\dOm)\big)^*,\quad s\in [0,1],
\end{equation} 
such that 
\begin{align}\lb{3.3}
\begin{split} 
& \langle f,g\rangle_s=\int_{\dOm} d^{n-1}\omega(\xi)\,\ol{f(\xi)} g(\xi),  
\\
& \text{whenever }\,\,f,g\in L^2(\dOm;d^{n-1}\omega),
\end{split} 
\end{align}
where $d^{n-1}\omega$ stands for the surface measure on $\dOm$.

One observes that the inclusion
\begin{equation}\label{inc-1}
\iota:H^{s_0}(\Omega)\hookrightarrow \bigl(H^r(\Omega)\bigr)^*,\quad
s_0>-1/2, \; r>1/2,
\end{equation} 
is well-defined and bounded. 

Next, we wish to describe a weak version of the normal trace operator associated
with $L$ in \eqref{L-def1}, considered in a bounded Lipschitz domain. 
To set the stage, assume Hypothesis \ref{h4.1} and introduce the weak Neumann trace operator 
\begin{equation}\label{2.8}
\wti\ga_N\colon\big\{u\in H^{s+1/2}(\Om)\,\big|\,Lu\in H^{s_0}(\Om)\big\}
\to H^{s-1}(\dOm),\quad s_0> -1/2,
\end{equation} 
as follows: Given $u\in H^{s+1/2}(\Om)$ with $Lu\in H^{s_0}(\Om)$
for some $s_0>-1/2$ and $s\in(0,1)$, we set (with $\iota$ as in
\eqref{inc-1} for $r:=3/2-s>1/2$) 
\begin{align}\label{2.9}
\langle\phi,\wti\ga_N u \rangle_{1-s} & := 
{}_{H^{1/2-s}(\Om)^n}\langle \nabla \Phi, A \nabla u\rangle_{(H^{1/2-s}(\Om)^n)^*}
\nonumber\\  
& \quad  - {}_{H^{3/2-s}(\Om)}\langle\Phi,\iota(Lu)\rangle_{(H^{3/2-s}(\Om))^*},
\end{align} 
for all $\phi\in H^{1-s}(\dOm)$ and $\Phi\in H^{3/2-s}(\Om)$ such that
$\ga_D\Phi=\phi$, where we denoted  
the Dirichlet trace operator by $\gamma_D$. We recall that 
\begin{equation}\label{2.8BBB}
\wti\ga_N\colon\big\{u\in H^{1}(\Om)\,\big|\,Lu\in H^{s_0}(\Om)\big\}
\to H^{-1/2}(\dOm),\quad s_0>-1/2,
\end{equation} 
is well-defined, linear, and bounded.

\vskip 0.08in

Moving on, we take up the task of describing the precise conditions 
that we impose on the nonlocal Robin boundary operator $\Theta$. 

%%%%%%%%%%%%%%%%%%%%%%%%%%%%%
\begin{hypothesis} \lb{h3.2}
Assume Hypothesis \ref{h4.1}, suppose that $\delta>0$ is a given number, and assume 
that $\Theta\in\cB\big(H^{1/2}(\dOm),H^{-1/2}(\dOm)\big)$
is a self-adjoint operator which can be written as 
\begin{equation}\label{Filo-1}
\Theta=\Theta^{(1)}+\Theta^{(2)}+\Theta^{(3)},
\end{equation} 
where $\Theta^{(j)}$, $j=1,2,3$, have the following properties: There exists 
a closed sesquilinear form $\gq_{\dOm}^{(0)}$ in $\LdOm$,
with domain $H^{1/2}(\dOm)\times H^{1/2}(\dOm)$, which is bounded from below
by $c_{\dOm}\in\bbR$ such that if $\Theta_{\dOm}^{(0)} \geq c_{\dOm}I_{\dOm}$ denotes
the self-adjoint operator in $\LdOm$ uniquely associated with $\gq_{\dOm}^{(0)}$, then 
$\Theta^{(1)}=\wti\Theta_{\dOm}^{(0)}$, the extension
of $\Theta_{\dOm}^{(0)}$ to an operator in $\cB\big(H^{1/2}(\dOm),H^{-1/2}(\dOm)\big)$. 
In addition,  
\begin{equation}\label{Filo-2}
\Theta^{(2)}\in\cB_{\infty}\big(H^{1/2}(\dOm),H^{-1/2}(\dOm)\big),
\end{equation} 
whereas $\Theta^{(3)}\in\cB\big(H^{1/2}(\dOm),H^{-1/2}(\dOm)\big)$ satisfies  
\begin{equation}\label{Filo-3}
\big\|\Theta^{(3)}\big\|_{\cB(H^{1/2}(\dOm),H^{-1/2}(\dOm))}<\delta.
\end{equation}
\end{hypothesis}
%%%%%%%%%%%%%%%%%%%%%%%%%%%%%%

The self-adjoint realization of the differential expression \eqref{L-def1} equipped
with nonlocal Robin type boundary conditions associated with an operator $\Theta$ 
as above is recorded below. 

%%%%%%%%%%%%%%%%%%%%
\begin{theorem} [\cite{GMN13}] \lb{t3.4}
Assume Hypothesis \ref{h3.2}, where the number $\delta>0$ is taken
to be sufficiently small relative to the Lipschitz character of $\Om$, more precisely, 
suppose that $0<\delta \leq \f{1}{6} \|\gamma_D\|^{-2}_{\cB(H^1(\Om),H^{1/2}(\dOm))}$.
In addition, consider the sesquilinear 
form $\gQ_{\Theta,\Om}(\dott,\dott)$ defined on $H^1(\Om) \times H^1(\Om)$ by 
\begin{align} \lb{3.8}
\begin{split} 
\gQ_{\Theta,\Om} (u,v):=\int_{\Om} d^nx\,
\bigl\langle\overline{(\nabla u)(x)},A(x)(\nabla v)(x)\bigr\rangle_{\bbC^n}
+\big\langle\gamma_D u,\Theta\gamma_D v \big\rangle_{1/2},&   \\
u,v\in H^1(\Om).&  
\end{split} 
\end{align} 
Then the form $\gQ_{\Theta,\Om}(\dott,\dott)$ 
in \eqref{3.8} is symmetric, $H^1(\Om)$-bounded, bounded from below, and closed 
in $L^2(\Om;d^n x)$. The self-adjoint operator $L_{\Theta,\Om}$ uniquely associated with 
$\gQ_{\Theta,\Om}$ on $L^2(\Om;d^nx)$ is then given by 
\begin{align}
& L_{\Theta,\Om}= L,    \lb{3.10BBB} \\
&  \dom(L_{\Theta,\Om})=
\big\{u\in H^1(\Om) \,\big|\,Lu\in L^2(\Om;d^nx),    
\, \big(\wti\gamma_N+\Theta\gamma_D\big)u=0 
\text{ in $H^{-1/2}(\dOm)$}\big\},   \no
\end{align} 
and is self-adjoint and bounded from below on $L^2(\Om;d^nx)$. Moreover, 
\begin{equation}\lb{3.11BBB}
\dom\big(|L_{\Theta,\Om}|^{1/2}\big)=H^1(\Om),
\end{equation} 
and $L_{\Theta,\Om}$, has purely discrete spectrum bounded from below. In particular, 
\begin{equation}\lb{3.12}
\sigma_{\rm ess}(L_{\Theta,\Om})=\emptyset.
\end{equation} 
\end{theorem}
%%%%%%%%%%%%%%%%%%%%

In the special case of Neumann boundary conditions (corresponding to $\Theta=0$), 
we use the notation
\begin{equation} \lb{3.13}
\gQ_{N,\Om}(\dott,\dott)= \gQ_{0,\Om}(\dott,\dott), \quad L_{N,\Om}= L_{0,\Om}.
\end{equation}

Next, we briefly comment on the usual case of a local Robin boundary condition, 
that is, the scenario when $\Theta$ is the operator $M_{\theta}$, of pointwise 
multiplication by a real-valued function $\theta$ defined on $\dOm$: 

%%%%%%%%%%%%%%%%%%%
\begin{lemma} [\cite{GM09}] \lb{l3.6}
Assume Hypothesis \ref{h4.1} and suppose that $\Theta=M_\theta$, the operator
of pointwise multiplication by a real-valued function 
$\theta\in L^p(\dOm;d^{n-1}\omega)$, where
\begin{equation}\label{Fpp}
p=n-1 \,\text{ if } \, n>2,  \, \text{ and } \, p\in(1,\infty] \, \text{ if } \, n=2. 
\end{equation}
Then
\begin{equation}\label{FGN-13}
\Theta\in\cB_{\infty}\big(H^{1/2}(\dOm),H^{-1/2}(\dOm)\big)
\end{equation}
is a self-adjoint operator which satisfies
\begin{equation}\label{FGN-14}
\|\Theta\|_{\cB(H^{1/2}(\dOm),H^{-1/2}(\dOm))}
\leq C\|\theta\|_{L^p(\dOm;d^{n-1}\omega)},
\end{equation}
for some finite constant $C=C(\Om,n,p)\geq 0$. In particular, the present situation 
$\Theta=M_\theta$ subordinates to the case $\Theta^{(2)}$ described in \eqref{Filo-2}.
\end{lemma}
%%%%%%%%%%%%%%%%%%%%

The $L^2$-realization of $L$ equipped with a Dirichlet boundary condition, 
$L_{D,\Om}$, in $L^2(\Om;d^n x)$ formally corresponds to $\Theta = \infty$. Note that
\begin{align}
& L_{D,\Om}= L,   \no \\
&  \dom(L_{D,\Om})=
\big\{u\in H^1(\Om)\,\big|\, Lu \in L^2(\Om;d^n x), \,
\gamma_D u =0 \text{ in $H^{1/2}(\dOm)$}\big\}    \lb{3.14BBB} \\
& \hspace*{1.8cm} = \big\{u\in H_0^1(\Om)\,\big|\, Lu \in L^2(\Om;d^n x)\big\}. \no 
\end{align} 

The well-known  Beurling-Deny criteria (cf. \cite{Da89}, \cite{Ou05}) allow to prove positivity preserving for the semigroup (and, equivalently, the resolvent) of 
$L_{\Theta, \Omega}$.  In order to achieve this, one assumes that
\begin{equation}
\big\langle\gamma_D |u|,\Theta  \gamma_D |u| \big\rangle_{1/2} \leq 
\big\langle\gamma_D u,\Theta \gamma_D u \big\rangle_{1/2},
\quad u \in H^1(\Om).   \lb{3.18} 
\end{equation}
Under this assumption, one has for $u \in H^1(\Omega)$, 
\begin{equation} 
\gQ_{\Theta,\Om} (| u |, | u |) \le \gQ_{\Theta,\Om} (u,u),
\end{equation} 
which by the first Beurling-Deny criterion is equivalent to positivity preserving of $e^{-t L_{\Theta, \Omega}}$.  
It is well-known  that positivity preserving is valid for $e^{-t L_{D, \Omega}}$ and $e^{-t L_{N, \Omega}}$.

%%%%%%%%%%%%%%%%%%%%%%%%%
\section{Gaussian Bounds}  \lb{s2}
%%%%%%%%%%%%%%%%%%%%%%%%%

Retaining Hypothesis \ref{h4.1} throughout this section, we now continue the discussion on divergence form elliptic partial differential operators $L_{\Theta,\Om}$ with nonocal Robin boundary conditions and focus on (Gaussian) heat kernel and Green's function bounds for $L_{\Theta,\Om}$.  

We will use the following heat kernel notation (for $t > 0$, a.e.\ $x, y \in \Om$)
\begin{align}
\begin{split}
&K_{\Theta,\Om} (t,x,y) =  e^{- t L_{\Theta, \Omega}}(x,y),  \quad 
K_{N,\Om} (t,x,y) =  e^{- t L_{N, \Omega}}(x,y),      \\
&K_{D,\Om} (t,x,y) =  e^{- t L_{D, \Omega}}(x,y),       
\end{split}
\end{align}
and similarly for Green's functions (for $z \in\bbC\backslash \bbR$, a.e.\ $x, y \in \Om$), 
\begin{align}
&G_{\Theta,\Om} (z,x,y) = (L_{\Theta, \Omega} - z I_{\Om})^{-1}(x,y),    \quad 
G_{N,\Om} (z,x,y) = (L_{N, \Omega} - z I_{\Om})^{-1}(x,y),    \no \\
&G_{D,\Om} (z,x,y) = (L_{D, \Omega} - z I_{\Om})^{-1}(x,y), \quad x \neq y. 
\end{align}

We recall that for $v \in L^2(\Omega;d^nx)$, $\overline{v}$ denotes 
the complex conjugate of $v$, and for two functions $u$ and $v$, 
the symbol $u.\overline{v} \ge 0$ means that the product of the functions, $u \overline{v}$, is 
nonnegative a.e.\ on $\Omega$.  

To state our the results of this section we need a few preparations: 

Let $\mathfrak{a}$ and $\mathfrak{b}$ be two sesquilinear, accretive, and closed forms
on $H = L^2(M;d\mu)$, and denote by $e^{-tA} $ and $e^{-tB}$ their associated semigroups, respectively. 

%%%%%%%%%%%%
\begin{theorem}\lb{t4.2}
Suppose that the semigroup $e^{-tB}$ is positivity preserving and that 
$\dom(\mathfrak{a}) = \dom(\mathfrak{b})$. Then the following assertions are equivalent.  \\[1mm] 
$(i)$ $| e^{-tA} f | \le e^{-t B} |f |$, $f\in L^2(M;d\mu)$, $t\geq 0$. \\[1mm]
$(ii)$ $ \Re (\mathfrak{a}(u,v)) \ge \mathfrak{b}(|u|, |v|)$ for all $u, v \in \dom(\mathfrak{a})$ such that 
$u.\overline{v} \ge 0$.\\[1mm] 
If both semigroups $e^{-tA} $ and $e^{-tB}$ are positivity preserving then assertion $(i)$ is equivalent 
to \\[1mm] 
$(iii)$ $ \Re (\mathfrak{a}(u,v)) \ge \mathfrak{b}(u, v)$ for all nonnegative  $u, v \in \dom(\mathfrak{a})$. 
\end{theorem}
%%%%%%%%%%%%
\begin{proof}[Sketch of proof] Since the semigroup $e^{-tB}$ is positivity preserving it follows from  
\cite[Proposition\ 2.20]{Ou05} that $\dom(\mathfrak{a})$ is an ideal of itself and hence an ideal of 
$\dom(\mathfrak{b})$ because $\dom(\mathfrak{a}) = \dom(\mathfrak{b})$ by hypothesis. We refer to  
\cite{Ou96} and \cite[Definition\ 2.19]{Ou05} for the notion of an ideal in this context. Thus, the 
equivalence of items $(i)$ and $(ii)$ follows from \cite[Corollary\ 3.4]{Ou96} (see also 
\cite[Theorem\ 2.21]{Ou05}). If both semigroups are positivity preserving, the equivalence of items 
$(i)$ and $(iii)$ follows from \cite[Theorem\ 3.7]{Ou96} (see also \cite[Theorem\ 2.24]{Ou05}). 
\end{proof} 
%%%%%%%%%%%%

Other criteria for the domination property in terms of forms in assertion $(i)$ for the case where  
$\dom(\mathfrak{a}) \neq \dom(\mathfrak{b})$ are given in \cite{Ou96} and \cite[Ch.\ 2]{Ou05}. The 
equivalence of items $(i)$ and $(iii)$ is also proved in \cite{GMN13}.  

We have the following result: 

%%%%%%%%%%% 
\begin{theorem}\lb{t4.3}
Assume Hypothesis \ref{h4.1}, suppose that $\Theta_j$, $j=1,2$, satisfy the assumptions 
introduced in Hypothesis \ref{h3.2}, and denote by $L_{\Theta_j,\Om}$ the operators in 
\eqref{3.10BBB} associated with the sesquilinear forms  $\gQ_{\Theta_j,\Om} (\dott,\dott)$, 
$j=1,2$, defined on $H^1(\Om)\times H^1(\Om)$ according to \eqref{3.8}. Suppose, in addition, that 
$\Theta_1$ satisfies \eqref{3.18} and that 
\begin{equation}
\Re \big(\big\langle\gamma_D u,\Theta_2  \gamma_D v \big\rangle_{1/2}\big) \geq 
\big\langle\gamma_D |u|,\Theta_1 \gamma_D |v|  \big\rangle_{1/2}, \lb{3.100}
\end{equation}
for all $u, v \in H^1(\Omega)$ such that  $u.\overline{v} \ge 0$. Then $e^{-t L_{\Theta_2, \Omega}}$ is dominated by $e^{-t L_{\Theta_1, \Omega}}$,
in the sense that
\begin{equation}
| e^{-t L_{\Theta_2, \Omega}} f | \le e^{-t L_{\Theta_1, \Omega}} |f |, \quad f\in L^2(\Omega;d^nx), 
\; t\geq 0.     \lb{3.101}
\end{equation}
If in addition $\Theta_1 \ge 0$, then 
\begin{equation}
| e^{-t L_{\Theta_2, \Omega}} f | \le e^{-t L_{\Theta_1, \Omega}} |f | \le e^{- t L_{N,\Omega}} |f|, 
\quad f\in L^2(\Omega;d^nx),\; t\geq 0.       \lb{3.102}
\end{equation}
\end{theorem}
%%%%%%%%%%%%
\begin{proof} 
We have seen at the end of Section \ref{s11} that  $e^{-tL_{\Theta_1, \Omega}}$ is positivity 
preserving. In addition, the forms $\gQ_{\Theta_2,\Om}$ and $\gQ_{\Theta_1,\Om}$ have the same domain $H^1(\Omega)$.  We are now in a position to apply Theorem \ref{t4.2}. 
One notes that $u \in H^1(\Omega)$ implies $| u | \in H^1(\Omega)$, and that 
\begin{equation} \label{dku}
\partial_k | u | = \Re \big( \partial_k u\cdot \text{sign}( \overline {u})\big), \quad 1 \leq k \leq n,
\end{equation} 
where  
\begin{equation}
\text{sign}( \overline {u})(x) = \begin{cases}
\frac{\overline{u(x)}}{| u(x) |}, & \text{if}\ u(x) \not= 0 \\[1mm] 
0, & \text{if}\ u(x) = 0.
\end{cases}
\end{equation} 
Formula \eqref{dku} is well-known (see, e.g.,  \cite[p.\ 104-105]{Ou05}). Using \eqref{dku} one concludes that 
\begin{equation} 
\Re \bigg(\int_{\Om} d^nx\,
\bigl\langle\overline{(\nabla u)(x)},A(x)(\nabla v)(x)\bigr\rangle_{\bbC^n}\bigg) \ge \int_{\Om} d^nx\,
\bigl\langle (\nabla |u|)(x),A(x)(\nabla |v|)(x)\bigr\rangle_{\bbC^n}, \lb{3.1002}
\end{equation}
for $u, v \in H^1(\Omega)$ such that $u.\overline{v} \ge 0$. 
Using \eqref{3.1002} and assumption \eqref{3.100} one infers that assertion $(ii)$ of Theorem \ref{t4.2} holds. An application of Theorem \ref{t4.2} then yields that \eqref{3.101} is satisfied. 

Similarly, again by Theorem \ref{t4.2}, the second inequality in \ref{3.102} holds once we prove that 
\begin{equation}
 \gQ_{\Theta_1,\Om} (u, v) \ge \gQ_{0,\Om} (u,v),
\lb{3.1001}
\end{equation}
for all nonnegative $u, v \in H^1(\Omega)$.  
This inequality follows along the same ideas as above, incorporating the assumption $\Theta_1 \ge 0$. \\
\end{proof}
%%%%%%%%%%

%%%%%%%%%%
\begin{remark} \lb{r4.51} 
The same proof shows that $e^{-tL_{D, \Omega}}$ is dominated by $e^{- t L_{\Theta_1, \Omega}}$ if $\Theta_1 \ge 0$.  This domination  is also stated  explicitly in \cite{GMN13}. 
\end{remark}
%%%%%%%%%%

%%%%%%%%%%%%%
\begin{remark} \lb{r4.5}
$(i)$ Under the hypotheses of Theorem \ref{t4.3}, all semigroups  $e^{-tL_{D,\Omega}}$,  $e^{-tL_{N,\Omega}}$ and $e^{- t L_{\Theta_1, \Omega}}$ for  $\Theta_1 \ge 0$ are sub-Markovian 
and hence extend to contraction semigroups on $L^\infty(\Om; d^nx)$. In addition, if  $\Theta_2$ is as in \eqref{3.100} then  $e^{- t L_{\Theta_2, \Omega}}$
extends to a contraction semigroup on $L^\infty(\Om; d^nx)$. 
Moreover, all these  semigroups  extend to strongly continuous semigroups on $L^p(\Om; d^nx)$, $p \in [1,\infty)$ (holomorphic semigroups for $p \in (1,\infty)$ in appropriate sectors), see \cite[Chs.\ 2, 3, 7]{Ou05}. \\
$(ii)$ Condition \eqref{3.18} is automatically satisfied in the special case of 
local Robin boundary conditions considered in Lemma \ref{l3.6}. \\
$(iii)$ One can add a potential $ 0 \leq V \in L^1_{\loc}(\Om; d^nx)$ to all operators in 
Theorem \ref{t4.3} by employing a standard sesquilinear form approach described in 
\cite[Remark\ 4.8]{GMN13}. 
\end{remark}
%%%%%%%%%%%%%

Now we  turn to the principal new results in this note. 

%%%%%%%%%%%%%
\begin{theorem} \lb{t4.12} Let $\Omega$ be a connected bounded Lipschitz domain in $\R^n$. Assume Hypothesis \ref{h4.1}, suppose that $\Theta$ satisfies  the assumptions 
introduced in Hypothesis \ref{h3.2}, and that
\begin{equation}
\Re \big(\big\langle\gamma_D u,\Theta  \gamma_D v \big\rangle_{1/2}\big) \geq 
0,  \lb{3.100-1}
\end{equation}
for $u, v \in H^1(\Omega)$ with $u. \overline{v} \ge 0$. Then there exist 
finite constants $C > 0$, $ c > 0$ such that $($for $t > 0$, a.e.\ $x, y \in \Om$$)$
\begin{equation}
| K_{\Theta,\Om} (t,x,y) | \leq C \max \big(t^{-n/2}, 1\big) \exp\big[- c |x-y|^2/  t\big].
\lb{3.1006}
\end{equation}
In addition, assuming 
\begin{equation} 
\langle \gamma_D 1, \Theta \, \gamma_D 1 \rangle_{1/2} \not=  0,   \lb{0}
\end{equation} 
where $1$ denotes the constant function with value $1$ on $\Omega$,  
then 
\begin{equation}\label{LLLam} 
\lambda_{1,\Theta,\Omega} = \inf \sigma(L_{\Theta, \Om}) > 0,  
\end{equation} 
and there exist finite constants $C >0$, $c > 0$, 
such that the Robin heat kernel $K_{\Theta, \Om}(t,\cdot,\cdot)$ satisfies 
$($for $t > 0$, a.e.\ $x, y \in \Om$$)$,
\begin{equation} \lb{est}
| K_{\Theta, \Om}(t,x,y)| \le C t^{-n/2} (1+ t)^{n/2} e^{- \lambda_{1,\Theta,\Omega}  t} 
\exp \big[- c |x-y|^2/t\big].  
\end{equation} 
\end{theorem} 
%%%%%%%%%%%%%
\begin{proof} 
On one hand, Theorem \ref{t4.3} and observation \eqref{2.13d} 
imply the following comparison for the Robin and Neumann heat kernels (for $t>0$, a.e.\ $x,y\in\Omega$)
\begin{equation}
| K_{\Theta,\Om} (t,x,y) | \leq K_{N,\Om} (t,x,y).        \lb{3.1003}
\end{equation} 
On the other hand, it is known that on a  bounded Lipschitz domain $\Omega$, 
the Neumann heat kernel $K_{N,\Om} (t,x,y)$ enjoys the Gaussian upper bound 
(for $t > 0$, a.e.\ $x, y \in \Om$) 
\begin{equation}
K_{N,\Om} (t,x,y) \leq C \max \big(t^{-n/2}, 1\big) \exp\big[- c |x-y|^2/  t\big],
\lb{3.1004}
\end{equation}
where $C$ and $c$ are positive finite constants. 
Combining \eqref{3.1004} with \eqref{3.1003}, one obtains \eqref{3.1006}.
%%%%%%%
By \eqref{3.12} (i.e., the compact embedding of $H^1(\Om)$ into $L^2(\Om)$), $L_{\Theta, \Om}$ 
has purely discrete spectrum. Let $\lambda_{1,\Theta,\Omega}:= \inf \sigma(L_{\Theta, \Om})$. Then $\lambda_{1,\Theta,\Omega}$ is 
the smallest eigenvalue of $L_{\Theta,\Omega}$ and we claim that
$\lambda_{1,\Theta,\Omega}$ is strictly positive. To justify this claim, we reason by 
contradiction and note that if $\lambda_{1,\Theta,\Omega} = 0$ then 
hypothesis \eqref{3.100-1} (with $u = v$)  and the fact that 
$L_{\Theta, \Om} \geq 0$ (cf.\ \eqref{3.8}) would imply the existence of a nonzero 
$u \in{\rm dom}(L_{\Theta, \Om})$ such that 
\begin{equation} 
(u, L_{\Theta, \Om} u)_{L^2(\Omega; d^nx)} 
+ \langle \gamma_D u, \Theta \gamma_D u \rangle_{1/2} = 0.
\end{equation}
Since $\langle \gamma_D u, \Theta \gamma_D u \rangle_{1/2} \ge 0$ by  \eqref{3.100-1}, 
 one concludes that $(u, L_{\Theta, \Om} u)_{L^2(\Omega; d^nx)} = 0$ and 
$ \langle \gamma_D u, \Theta \gamma_D u \rangle_{1/2} = 0$. The first equality together with the 
assumed ellipticity of $L$ implies that $u$ is constant on $\Omega$. The second equality together with \eqref{0} 
then yields the desired contradiction. This proves \eqref{LLLam}.
%%%%
Next, we improve on \eqref{3.1006} to obtain \eqref{est}. Obviously, we  may consider $t \ge 1$, only. 
By \cite[Lemma\ 6.5]{Ou05} and \eqref{3.1006} we obtain (for $t > 0$, a.e.\ 
$x, y \in \Om$) 
\begin{equation}\label{2}
| K_{\Theta, \Om}(t,x,y) |  \le C  t^{-n/2} e^{-\lambda_{1,\Theta,\Omega}  t} 
[1+ \lambda_{1,\Theta,\Omega} t]^{n/2}.
\end{equation}
Now since $\Omega$ is bounded it has finite diameter. Therefore, 
$|x-y| \le {\rm diam}\,(\Omega)$ for all $x, y \in \Omega$. From \eqref{2} 
we may estimate that (for $t \geq 1$, a.e.\ $x, y \in \Om$) 
\begin{align}
| K_{\Theta, \Om}(t,x,y) | & \leq  C  t^{-n/2} e^{-\lambda_{1,\Theta,\Omega}  t} 
[1+ \lambda_{1,\Theta,\Omega} t]^{n/2}   \no \\
&= C  t^{-n/2} e^{-\lambda_{1,\Theta,\Omega}  t} e^{-c | x - y |^2 / t} e^{c | x - y |^2 / t} 
[1+ \lambda_{1,\Theta,\Omega} t]^{n/2}    \no \\
&\leq C  t^{-n/2} e^{-\lambda_{1,\Theta,\Omega}  t} e^{-c | x - y |^2 / t} e^{c[{\rm diam}\,(\Omega)]^2} 
[1+ \lambda_{1,\Theta,\Omega} t]^{n/2}, 
\end{align}
completing the proof. 
\end{proof} 
%%%%%%%%%%%%%%

%%%%%%%%%%%%%%
\begin{remark} \lb{r4.13} 
$(i)$  From Theorem \ref{t4.12} one  obtains  the following estimates for the Robin Green function. 
For  $\lambda > 0$, and a.e.\ $x, y \in \Om$,
\begin{align}
| G_{\Theta,\Om} (\lambda,x,y) | 
& \leq \begin{cases} C_{a_0,a_1,\lambda,\Omega,n} |x - y|^{2-n}, & n \geq 3, 
\\[6pt]
C_{a_0,a_1,\lambda,\Omega} \big|\ln\big(1 + |x - y|^{-1}\big)\big|, & n=2,  \end{cases} 
\quad x \neq y.    
\end{align}
This follows as usual by writing the resolvent as  the Laplace transform of the semigroup and hence the Green function at $\lambda$ as the Laplace transform of the heat kernel (cf.\ \cite[App.\ C]{GMN13}). \\[1mm] 
$(ii)$  The semigroup $e^{-t L_{\Theta, \Om}}$ is bounded holomorphic on $L^p(\Omega;d^nx)$ 
in the sector $\{z \in \bbC \,|\, |\arg(z)| < \pi/2\}$ for all $p \in [1, \infty)$. In particular, 
the generator of the corresponding semigroup has (minus) spectrum contained in $[0, \infty)$ and is 
$p-$independent. See \cite[Ch.\ 7]{Ou05}. \\[1mm] 
$(iii)$  The operator $L_{\Theta, \Om}$ has a bounded holomorphic functional calculus and one even has a spectral multiplier result, see, \cite[Theorem\ 7.23]{Ou05}. 
\end{remark}
%%%%%%%%%%%%%%

Define the metric $\rho(\cdot,\cdot)$ on $\Omega$ by setting, for each $x,y\in\Omega$,
\begin{align}
\begin{split} 
&\rho(x,y):= \sup\bigg\{\varphi(x) - \varphi(y) \, \bigg| \, 
\varphi \in W^{1,\infty}(\bbR^n),\,\,\text{ real-valued, and} \\ 
& \hspace*{3cm} \sum_{j,k = 1}^n a_{j,k} (\xi) \partial_j \varphi(\xi) \partial_k \varphi (\xi) \le 1 
\text{ for a.e.\ $\xi\in\Omega$} \bigg\}.
\end{split} 
\end{align}

This is the metric associated with the coefficients $a_{j,k}$. By  ellipticity, 
$\rho(\cdot,\cdot)$ is clearly (two-sided, pointwise) comparable with the standard Euclidean metric. 
Following the method in \cite{Ou06} for Schr\"odinger-type operators one can actually derive 
a sharper  estimate using $\rho(\cdot,\cdot)$ instead of  the Euclidean metric. 

%%%%%%%%%%%%
\begin{proposition}\label{prS}
Suppose that the assumptions of Theorem \ref{t4.12} are satisfied.  In addition, 
assume that
\begin{equation}\label{4}
\big\langle \gamma_D(e^{\varphi} u), \Theta \, \gamma_D (e^{-\varphi} u) \big\rangle_{1/2}  \geq 
\langle \gamma_D u, \Theta \, \gamma_D u \rangle_{1/2}
\end{equation}
for all real-valued $\varphi \in C^{\infty}(\bbR^n) \cap L^{\infty}(\bbR^n)$. 
Then there exist finite constants $C > 0$, $c > 0$ such that the Robin heat kernel
$K_{\Theta, \Om}(t,\cdot,\cdot)$ 
satisfies $($for $t > 0$, a.e.\ $x, y \in \Om$$)$, 
\begin{align}\label{10}
\begin{split} 
| K_{\Theta, \Om}(t,x,y) |  & \le C' t^{-n/2} e^{-\lambda_{1,\Theta,\Omega} t} 
\big[1 + \lambda_{1,\Theta,\Omega} t +   [\rho(x,y)^2/t \big]^{n/2}   \\
& \quad \times \exp \big[- \rho(x,y)^2/(4t)\big]. 
\end{split} 
\end{align}
\end{proposition}
%%%%%%%%%%%%%
\begin{proof}
In order to prove estimate \eqref{10}, assuming condition \eqref{4}, we fix 
$\lambda  \in \bbR$ and a real-valued function $\varphi \in W^{1, \infty}(\bbR^n)$ 
such that 
\begin{equation}\label{5}
\sum_{j,k=1}^n a_{j,k} (x) \partial_j \varphi(x) \partial_k \varphi (x) \leq 1 \, 
\text{ for a.e.\ $x\in \Omega$}.
\end{equation}
Following Davies' perturbation method (see, e.g., \cite[Ch.\ 3]{Da89}),
introduce $S_\lambda(t) := e^{\lambda \varphi} 
e^{-t L_{\Theta, \Om}} e^{-\lambda \varphi} $. This semigroup has 
integral kernel given by 
\begin{equation} 
e^{\lambda( \varphi(x) - \varphi(y))}  K_{\Theta, \Om}(t,x,y)\ \text{ a.e. $x,y\in \Omega$, $t>0$.}
\end{equation} 
Using Theorem \ref{t4.12}, one obtains for any $\delta\in(0,1)$ there exist 
finite constants $c_\delta,C_\delta>0$ such that (for $t>0$, a.e.\ $x,y \in \Om$), 
\begin{equation}\label{6}
e^{\lambda [\varphi(x) - \varphi(y)]} | K_{\Theta, \Om}(t,x,y) | \leq e^{\frac{1}{\delta} \lambda^2 t} 
e^{\delta [\varphi(x) -\varphi(y)]^2/t} C_{\delta} t^{-n/2}  \exp \big[- c_{\delta} |x-y|^2/t\big].  
 \end{equation}
Choosing $\delta > 0$ sufficiently small and using \eqref{5} one obtains that, 
on the one hand (for $t>0$, a.e.\ $x,y \in \Om$), 
\begin{equation}\label{7}
e^{\lambda [\varphi(x) - \varphi(y)]}  | K_{\Theta, \Om}(t,x,y)|  \le e^{\frac{1}{\delta} \lambda^2 t} 
 C t^{-n/2}.
 \end{equation}
On the other hand, the semigroup $S_\lambda(t)$ is associated with the sesquilinear form 
\begin{equation} 
\gQ_{\Theta, \Om}\big(e^{-\lambda \varphi}u, e^{\lambda \varphi}v\big), \quad u, v \in H^1(\Om), 
\end{equation} 
where $\gQ_{\Theta, \Om}$ is the form of $L_{\Theta, \Om}$. One verifies that 
\begin{align}
\gQ_{\Theta, \Om}\big(e^{-\lambda \varphi}u, e^{\lambda \varphi}u\big) 
&= \sum_{j,k = 1}^n \int_\Om d^n x \, a_{j,k} \overline{\partial_j u} \partial_k u 
- \lambda^2 \sum_{j,k = 1}^n \int_\Om d^n x \, a_{j,k} \partial_j \varphi \partial_k  \varphi |u |^2    \no \\
& \quad + \langle \gamma_D (e^{-\lambda \varphi}u),  \Theta \gamma_D (e^{-\varphi} u) \rangle_{1/2}, 
\quad u \in H^1(\Om).
\end{align}
Thus, using \eqref{5} and  \eqref{4},
\begin{align}  
&\Re \big(\gQ_{\Theta, \Om}\big(e^{-\lambda \varphi}u, e^{\lambda \varphi}u\big)\big) \geq 
\gQ_{\Theta, \Om}(u,u) - \lambda^2 \int_\Om d^n x \, | u |^2 
\geq (\lambda_{1,\Theta,\Omega} - \lambda^2)\int_\Om d^n x \, | u |^2,     \no \\[1mm] 
& \hspace*{9.2cm} u\in H^1(\Omega).
\end{align} 
The latter inequality implies the following estimate 
\begin{equation}\label{8}
\| S_\lambda(t) \|_{\cB(L^2(\Omega; d^nx))} \le e^{- (\lambda_{1,\Theta,\Omega} -\lambda^2)t}.
\end{equation} 
Now we proceed as in the beginning of the poof of Theorem \ref{t4.12}. 
Estimates \eqref{7}, \eqref{8}, and \cite[Lemma\ 6.5]{Ou05} imply (for $t>0$, a.e.\ $x,y \in \Om$), 
\begin{equation}\label{9}
e^{\lambda [\varphi(x) - \varphi(y)]}  | K_{\Theta, \Om}(t,x,y)|  \le Ct^{-n/2} e^{-\lambda_{1,\Theta,\Om} t} e^{\lambda^2 t} \big[1 + \lambda_{1,\Theta,\Om} t + \delta' \lambda^2 t\big]^{n/2},
\end{equation}
with $\delta' := \delta^{-1} -1$. We arrive at (for $t>0$, a.e.\ $x,y \in \Om$), 
\begin{equation} 
| K_{\Theta, \Om}(t,x,y) |   \le C' t^{-n/2} e^{-\lambda_{1,\Theta,\Om} t} e^{\lambda^2 t} 
\big[1 + \lambda_{1,\Theta,\Om} t +  \lambda^2 t\big]^{n/2} e^{ - \lambda [\varphi(x) - \varphi(y)]}.
\end{equation} 
Choosing $\lambda = [\varphi(x) - \varphi(y)]/(2t)$ and optimizing over $\varphi$ yields \eqref{10}. 
\end{proof}
%%%%%%%%%%%%%

One notes that while condition \eqref{4} may not be automatically satisfied  
in the presence of nonlocal Robin boundary conditions, it is certainly 
fulfilled in the case of local Robin boundary conditions.

%%%%%%%%%%%%%%%%%%%%%%%%%%%%%%%%%%%%%%
%%%%%%%%%%%%%%%%%%%%%%%%%%%%%%%%%%%%%%
\section{Some Illustrations}  \lb{s3}
%%%%%%%%%%%%%%%%%%%%%%%%%%%%%%%%%%%%%%
%%%%%%%%%%%%%%%%%%%%%%%%%%%%%%%%%%%%%%

We conclude this note with a number of concrete examples illustrating Theorem \ref{t4.12}. 

Assuming throughout this section that $n\geq 3$, one recalls that 
\begin{equation}\label{Riah-1}
E_n(x):= \frac{1}{(n-2)\omega_{n-1}}\frac{1}{|x|^{n-2}},   \quad 
x\in{\mathbb{R}}^n \backslash \{0\},
\end{equation} 
is the fundamental solution for (minus) the Laplacian $- \Delta$ in ${\mathbb{R}}^n$. Here 
$\omega_{n-1}=2\pi^{n/2}/\Gamma(n/2)$ ($\Gamma(\dott)$ the Gamma function, 
cf.\ \cite[Sect.\ 6.1]{AS72}) represents the area of the unit sphere $S^{n-1}$ in $\bbR^n$. 

Next, suppose that $\Omega\subseteq{\mathbb{R}}^n$ is a Lipschitz domain with compact 
boundary, and denote by $\omega$ the canonical surface measure on $\partial\Omega$. 
Then the (boundary-to-boundary version of the) harmonic single layer associated
with $\Omega$ is the integral operator of formal convolution with $E_n$, that is, 
\begin{equation}\label{Riah-2}
(Sf)(\xi):= - \int_{\partial\Omega} d^{n-1} \omega (\eta) \, E_n(\xi - \eta)f(\eta),  \quad \xi \in\partial\Omega.
\end{equation} 
One observes that in the special case where $\partial\Omega\in C^\infty$, it follows that $S$ is a 
classical pseudodifferential operator of order $-1$. This description of $S$ 
is tightly connected with the strong regularity assumption on the boundary of $\Omega$,
and fails to materialize in the presence of just one boundary irregularity. 
Nonetheless, $S$  continues to enjoy remarkable properties even when considered 
on rough surfaces, as in the presently assumed Lipschitz setting. 
Some of its basic properties relevant for us here are as follows:
\begin{align}\label{Riah-3}
\begin{split} 
& S:L^2(\partial\Omega; d^{n-1} \omega)\longrightarrow L^2(\partial\Omega; d^{n-1} \omega)
\, \text{ is linear, compact}, \\
& \quad \text{nonnegative, injective, with range $H^1(\partial\Omega)$}.
\end{split}
\end{align}
Functional calculus then yields that 
\begin{equation}\label{Riah-4}
\text{$S^{i\gamma}$ is a unitary operator on $L^2(\partial\Omega; d^{n-1} \omega)$  
for each $\gamma\in{\mathbb{R}}$.}
\end{equation} 
Also, starting from the fact that  
\begin{equation}\label{Riah-5}
S:L^2(\partial\Omega; d^{n-1} \omega)\longrightarrow H^1(\partial\Omega)
\, \text{ is a linear, bounded, isomorphism},
\end{equation} 
by duality and interpolation we obtain that  
\begin{equation}\label{Riah-6}
S:H^{s-1}(\partial\Omega)\longrightarrow  H^s(\partial\Omega)
\, \text{ is a linear, bounded, isomorphism.} 
\end{equation} 
Functional calculus may be also used to define complex and fractional powers of $S$.

%%%%%%%%%%%%%
\begin{lemma}\label{lUTD-1}
For each $\alpha\in[0,1]$ and $s\in[0,\alpha]$,   
\begin{equation}\label{Riah-7}
S^{-\alpha}:H^s(\partial\Omega)\longrightarrow  H^{s-\alpha}(\partial\Omega)
\, \text{ is a linear, bounded, isomorphism}.
\end{equation}
\end{lemma}
%%%%%%%%%%%%%
\begin{proof}
In a first stage, we shall show that   
\begin{align}\label{Riah-8}
\begin{split}
& S^{\alpha}:L^2(\partial\Omega; d^{n-1} \omega)\longrightarrow  H^\alpha(\partial\Omega)
\, \text{ is a linear, bounded}, \\ 
& \quad \text{isomorphism, for each $\alpha\in[0,1]$.}
\end{split}
\end{align} 
To justify this, note that the family of operators 
$\{S^z\}_{z \in \bbS}$ (with $\bbS = \{z\in \bbC \,|\, 0 \leq \Re(z) \leq 1\}$ denoting a closed 
complex strip), depends analytically on the parameter $z$ in the interior of $\bbS$,
when viewed as a mapping with values in $\cB\big(L^2(\partial\Omega; d^{n-1} \omega)\big)$,
and   
\begin{align}\label{Riah-9}
\begin{split}
& S^z:L^2(\partial\Omega; d^{n-1} \omega)\longrightarrow  H^{\Re(z)}(\partial\Omega)
\, \text{ is a linear, bounded},  \\ 
& \quad \text{isomorphism when $\Re(z)=0$ or $\Re(z)=1$,} 
\end{split}
\end{align} 
due to \eqref{Riah-4} and \eqref{Riah-5}. Then \eqref{Riah-8} follows
from this and Stein's complex interpolation theorem for analytic families
of operators. Having established \eqref{Riah-8}, via duality and interpolation, 
one then obtains that 
\begin{align}\label{Riah-10}
\begin{split}
& S^\alpha:H^{s-\alpha}(\partial\Omega)\longrightarrow  H^s(\partial\Omega)
\, \text{ is a linear, bounded}, \\
& \quad \text{isomorphism whenever $\alpha\in[0,1]$ and $s\in[0,\alpha]$.}
\end{split}
\end{align} 
Taking inverses, this finally yields \eqref{Riah-7}.
\end{proof}
%%%%%%%%%%%%%

The next step is to prove a nondegeneracy condition for fractional powers
of the harmonic single layer, in the spirit of \eqref{0}.

%%%%%%%%%%%%
\begin{lemma}\label{lUTD-2}
If $1$ denotes the constant function $1$ in $\Omega$, then for each $\alpha\in[0,1]$
one has 
\begin{equation}\label{Riah-12}
\big\langle\gamma_D 1,S^{-\alpha}\gamma_D 1 \big\rangle_{1/2}>0.
\end{equation}
\end{lemma}
%%%%%%%%%%%%
\begin{proof}
Since $\gamma_D 1\in H^s(\partial\Omega)$ for each $s\in[0,1]$, 
from \eqref{Riah-7} (used with $s=\alpha$) one deduces that   
\begin{equation}\label{Riah-13}
S^{-\alpha}\gamma_D 1\in L^2(\partial\Omega; d^{n-1} \omega) \, \text{ for each } \, 
\alpha\in[0,1].
\end{equation} 
Next, fix $\alpha\in[0,1]$. Then on account of the self-adjointness of $S$ 
and \eqref{Riah-13} one can write  
\begin{equation}\label{Riah-14}
\big\langle\gamma_D 1,S^{-\alpha}\gamma_D 1\big\rangle_{1/2}
= \big\langle S^{-\alpha/2}\gamma_D 1,S^{-\alpha/2}\gamma_D 1\big\rangle_{1/2}
=\int_{\partial\Omega} d^{n-1} \omega \, |S^{-\alpha/2}\gamma_D 1|^2.
\end{equation} 
Given that Lemma \ref{lUTD-1} ensures that the function 
$S^{-\alpha/2}\gamma_D 1$ is not identically zero on $\partial\Omega$, 
the claim in \eqref{Riah-12} now readily follows from \eqref{Riah-14}.
\end{proof}
%%%%%%%%%%%%

After this preamble, we are in a position to prove the following result
which identifies a class of highly nonlocal operators satisfying Hypothesis \ref{h3.2}
along with a nondegeneracy condition in the spirit of \eqref{0}. We recall that
$M_{\theta}$ stands for the operator of pointwise multiplication by the measurable 
function $\theta$. 

%%%%%%%%%%%%
\begin{theorem}\label{tYRS-1}
Assume $\theta: \partial\Omega \to [0,\infty)$ is a function that is strictly 
positive on a subset of $\partial\Omega$ of positive $d^{n-1} \omega$ measure 
and satisfies $\theta\in L^p(\dOm;d^{n-1}\omega)$, where
\begin{equation}\label{Fpp1}
p=n-1 \,\mbox{ if } \, n>2,  \mbox{ and } \, p\in(1,\infty] \,\mbox{ if } \, n=2. 
\end{equation}
Then for any given number $\delta>0$ there exists $\varepsilon > 0$ with the property 
that for each $\alpha\in[1/2,1)$ the operator   
\begin{equation}\label{Riah-15}
\Theta:= c_1 M_{\theta}+ c_2 S^{-\alpha} + c_3 \varepsilon S^{-1}, \quad 
c_j \geq 0, \; 1 \leq j \leq 3, \quad c_1 c_2 c_3 \neq 0, 
\end{equation} 
satisfies the conditions stipulated in Hypothesis \ref{h3.2} for the given $\delta$, 
as well as the nondegeneracy condition  
\begin{equation}\label{Riah-16}
\langle\gamma_D 1,\Theta\gamma_D 1\rangle_{1/2}\not=0.
\end{equation}
\end{theorem}
%%%%%%%%%%%%%
\begin{proof}
Decompose $\Theta=\Theta^{(1)}+\Theta^{(2)}+\Theta^{(3)}$, where  
\begin{equation}\label{Fpgc-1}
\Theta^{(1)}:= c_1 M_{\theta},\quad \Theta^{(2)}:= c_2 S^{-\alpha}, 
\quad\Theta^{(3)}:= c_3 \varepsilon S^{-1}.
\end{equation} 
From Lemma \ref{l3.6} and the fact that $\theta$ is positive on a subset of 
$\partial\Omega$ of positive $d^{n-1} \omega$ measure, it follows that   
\begin{align}\label{Riah-17}
\begin{split}
& \Theta^{(1)}\in\cB\big(H^{1/2}(\dOm),H^{-1/2}(\dOm)\big), \\
& \Theta^{(1)} \, \text{ is self-adjoint in this context and } \, 
\big\langle\gamma_D 1,\Theta^{(1)}\gamma_D 1\big\rangle_{1/2}>0 \, \text{ if $c_1 > 0$.}
\end{split}
\end{align} 
Next, since $S^{-\alpha}$ maps $H^{1/2}(\dOm)$ boundedly into $H^{1/2-\alpha}(\dOm)$ 
by Lemma \ref{lUTD-1} (used here with $s=1/2$ and $\alpha\in[1/2,1)$), and since  
$H^{1/2-\alpha}(\dOm)$ embeds compactly into $H^{1/2}(\dOm)$ given that we 
are currently assuming $\alpha<1$, it follows that  
\begin{equation}\label{Riah-18}
\Theta^{(2)}\in\cB_\infty\big(H^{1/2}(\dOm),H^{-1/2}(\dOm)\big).
\end{equation} 
Moreover, it is clear from the self-adjointness of $S$ that $\Theta^{(2)}$ is 
self-adjoint in the context of $\cB\big(H^{1/2}(\dOm),H^{-1/2}(\dOm)\big)$.
In addition, Lemma \ref{lUTD-2} yields   
\begin{equation}\label{Riah-19}
\big\langle\gamma_D 1,\Theta^{(2)}\gamma_D 1\big\rangle_{1/2}>0  \, \text{ if $c_2 > 0$.}
\end{equation} 
Similar considerations, based on Lemma \ref{lUTD-1} (used with $\alpha=1$ and $s=1/2$) 
and Lemma \ref{lUTD-2} (used with $\alpha=1$), prove that  
\begin{align}\label{Riah-20}
\begin{split}
& \Theta^{(3)}\in\cB\big(H^{1/2}(\dOm),H^{-1/2}(\dOm)\big), \\
& \Theta^{(3)} \, \text{ is self-adjoint in this context and } \, 
\big\langle\gamma_D 1,\Theta^{(3)}\gamma_D 1\big\rangle_{1/2}>0  \, \text{ if $c_3 > 0$.} 
\end{split}
\end{align} 
Finally, it remains to observe that, given any $\delta>0$, 
Lemma \ref{lUTD-1} may be invoked (with $\alpha=1$ and $s=1/2$) in order to 
find a number $\varepsilon>0$ small enough so that  
\begin{equation}\label{Fiyhi}
\big\|\Theta^{(3)}\big\|_{\cB(H^{1/2}(\dOm),H^{-1/2}(\dOm))}<\delta.
\end{equation} 
The above analysis then proves that $\Theta$ from \eqref{Riah-15}
satisfies all conditions in Hypothesis \ref{h3.2} as well as the 
nondegeneracy condition \eqref{Riah-16}.
\end{proof}
%%%%%%%%%%%%

Another, more elementary, example of a nonlocal operator satisfying Hypothesis \ref{h3.2}
as well as the nondegeneracy condition \eqref{0} may be produced as follows. We 
retain the assumption that $\Omega\subseteq{\mathbb{R}}^n$ is a Lipschitz domain
with compact boundary and consider a measurable kernel  
\begin{equation}\label{Fi-AA.1}
k:\partial\Omega\times\partial\Omega\to{\mathbb{C}}
\end{equation} 
satisfying the symmetry condition  
\begin{equation}\label{Fi-AA.2}
k(\xi,\eta)=\overline{k(\eta,\xi)} \quad \xi,\eta\in\partial\Omega,
\end{equation} 
and such that 
\begin{align}
& \int_{\partial\Omega}\int_{\partial\Omega} d^{n-1} \omega(\xi) d^{n-1} \omega (\eta) \, 
|k(\xi,\eta)| < +\infty,    \label{Fi-AA.3a} \\
& \int_{\partial\Omega} d^{n-1} \omega (\eta) \left|\int_{\partial\Omega} d^{n-1} \omega (\xi) \, 
k(\xi,\eta)\right|^2 \not=0. \label{Fi-AA.3b}
\end{align}   
Moreover, suppose that the kernel $k$ is sufficiently decent such that the integral operator  
\begin{equation}\label{Fi-AA.4}
(Af)(\xi):=\int_{\partial\Omega} d^{n-1} \omega (\eta) \, k(\xi,\eta)f(\eta),
\quad \xi \in \partial\Omega,
\end{equation} 
satisfies  
\begin{equation}\label{RTjn}
A\in\cB_\infty\big(H^{1/2}(\dOm),L^2(\dOm; d^{n-1} \omega)\big).
\end{equation} 
For example, condition \eqref{RTjn} holds if $A$ is compact on $L^2(\partial\Omega; d^{n-1} \omega)$, 
which is always the case if $k$ satisfies the stronger Hilbert--Schmidt condition 
\begin{equation}\label{Fi-AA.5}
\int_{\partial\Omega}\int_{\partial\Omega} d^{n-1} \omega (\xi) d^{n-1} \omega (\eta) \, 
|k(\xi,\eta)|^2 < +\infty
\end{equation} 
in place of condition \eqref{Fi-AA.3a}.
Assuming that \eqref{Fi-AA.4} holds, one then deduces that 
\begin{equation}\label{RTjn-2}
A^*\in\cB_\infty\big(L^2(\dOm; d^{n-1} \omega),H^{-1/2}(\dOm)\big) 
\end{equation} 
and, ultimately, that 
\begin{equation}\label{RTjn-3}
\Theta:=A^*A\in\cB_\infty\big(H^{1/2}(\dOm),H^{-1/2}(\dOm)\big). 
\end{equation} 
Moreover, the linear operator $\Theta$ defined in \eqref{RTjn-3} is self-adjoint 
in the context of $\cB\big(H^{1/2}(\dOm),H^{-1/2}(\dOm)\big)$ and  
\begin{equation}\label{Riah-21}
\langle\gamma_D 1,\Theta\gamma_D 1\rangle_{1/2}
=\int_{\partial\Omega} d^{n-1} \omega \, |A\gamma_D 1|^2 
=\int_{\partial\Omega} d^{n-1} \omega (\eta) \left|\int_{\partial\Omega} 
d^{n-1} \omega (\xi) \, k(\xi,\eta)\right|^2 \not = 0,
\end{equation} 
by condition \eqref{Fi-AA.3b}. The bottom line is that
that the operator $\Theta$ in \eqref{RTjn-3} satisfies Hypothesis \ref{h3.2}
(taking $\Theta^{(1)}=\Theta^{(3)}=0$), as well as the 
nondegeneracy condition \eqref{Riah-16}.

\medskip

%%%%%%%%%%%%%%%%%%%%%%%%%%%%%%%%%%%%%
%{\bf Acknowledgments.} 
%%%%%%%%%%%%%%%%%%%%%%%%%%%%%%%%%%%%%

%%%%%%%%%%%%%%%%%%%%%%%%%%%%%%%%
%%%%%%%%%%%%%%%%%%%%%%%%%%%%%%%%

\end{document}